\documentclass{amsart}
\usepackage{amssymb}
\usepackage{amsmath}
\usepackage{enumerate}

\title{On the interlacing of cylinder functions}

\author{T. P\'almai}
\address{Department of Theoretical Physics\\
  Budapest University of Technology and Economics\\
  H-1111 Budapest Hungary}
\email{palmai@phy.bme.hu}



\begin{document}

\newcommand\Real{\mathbb{R}}

\newcommand\sgn{\operatorname{sgn}}

\newtheorem{thm}{Theorem}
\newtheorem{lemma}[thm]{Lemma}
\newtheorem{cor}[thm]{Corollary}

\maketitle

\begin{abstract}
Necessary and sufficient conditions for the interlacing of the zeros of cylinder functions and their derivatives of different orders are given.
\end{abstract}

\pagestyle{myheadings}
\thispagestyle{plain}

\section{Introduction}

The zeros of the Bessel functions have been subjects of studies for more than a century; the field of applications is vast. In the monograph of Watson \cite{Watson} a number of aspects are discussed, a summary of the most important facts are listed in \cite{Abramowitz}. More recent results can be found in Refs. \cite{Elbert,Laforgia} where further references are given.
In this note I derive a new interlacing theorem for cylinder functions in the form of \emph{necessary and sufficient} conditions. The results are of primary interest from the point of view of the theory of Bessel functions; however, applications might also arise as in Refs. \cite{LZ2,PA1}, where such relations were useful in inverse scattering problems.

The general solution of the Bessel differential equation (up to a constant multiplier) is given by the cylinder function \cite{Watson}
\begin{equation}\label{eq:cyl}
C_\nu(x)\equiv J_\nu(x)\cos(\delta)-Y_\nu(x)\sin(\delta)
\end{equation}
where $J_\nu(x)$ and $Y_\nu(x)$ are the Bessel functions of the first and second kind, respectively.
Considering the Bessel differential equation, a second order linear homogeneous ODE, satisfied by the Bessel functions it is easy to see that $J_{\nu}(x)$, $Y_{\nu}(x)$ and $J'_{\nu}(x)$, $Y'_{\nu}(x)$ each has an infinity of real zeros, for any given real value of $\nu$. Furthermore, these zeros are all simple with the possible exception of $x=0$. I will use the term \emph{interlace} for two functions if between each consecutive pair of zeros of one function there is one and only one zero of the other. Denote the $s$th zero of the functions $J_{\nu}(x)$, $Y_{\nu}(x)$, $J'_{\nu}(x)$, $Y'_{\nu}(x)$, $C_\nu(x)$ and $C'_\nu(x)$ by $j_{\nu,s}$, $y_{\nu,s}$, $j'_{\nu,s}$, $y'_{\nu,s}$, $c_{\nu,s}$ and $c'_{\nu,s}$, respectively, except that $x=0$ is counted as the first zero of $J'_0(x)$ \cite{Watson}.

The following theorem summarizes some known relevant interlacing results.

\begin{thm}[\cite{Watson,Segura,LZ1,LZ2,PA1,PA2}]\label{thm:1} For $\nu\geq0$ the following points hold true.
\begin{enumerate}[\indent (a)]
\item  For $0<a\leq2$ the positive real zeros of $C_\nu(x)$ and $C_{\nu+a}(x)$ are interlaced. Similarly,  $J'_{\nu}(x)$, $J'_{\nu+b}(x)$ and $Y'_{\nu}(x)$, $Y'_{\nu+b}(x)$ are also interlaced if $0<b\leq1$, respectively.
\item If $0<c\leq1$ the inequality sequence
\begin{equation}
j'_{\nu,s}<y_{\nu,s}<y_{\nu+c,s}<y'_{\nu,s}<j_{\nu,s}<j_{\nu+c,s}<j'_{\nu,s+1}\quad s=1,2,\ldots
\end{equation}
holds. For $c>1$ this property is destroyed. We also have $\nu\leq j'_{\nu,1}$.
\end{enumerate}
\end{thm}

Note that this particular formulation of the interlacing results was obtained only recently \cite{Segura,LZ1,PA2}.

A very important fact (which is a consequence of the Watson formula \cite[p. 508 Eq. (3)]{Watson}) is stated in the following theorem.

\begin{thm}\label{thm:cont}
$c_{\nu,s}$ and $c'_{\nu,s}$ are continuous increasing functions of the order $\nu>0$ for all $s=1,2,\ldots$.
\end{thm}

\section{Results}

The main result is formulated as follows.

\begin{thm}\label{THM1}
For positive orders $\nu$ and $\mu$ the positive zeros of
$$
C_\nu(x),\, C_\mu(x);\qquad J'_\nu(x),\,J'_\mu(x);\qquad Y'_\nu(x),\,Y'_\mu(x)
$$
are interlaced, respectively, if and only if $|\nu-\mu|\leq2$.
\end{thm}

{\it Remarks.} In general, if at least one of $\nu$ and $\mu$ is negative $C_\nu(x)$ and $C_\mu(x)$ are not interlaced on $(0,\infty)$, i.e. not all the positive real zeros are interlaced, unless the zeros are defined as continuous increasing functions of the order (see Ref. \cite{Watson}, pp. 508-510 on how the zeros disappear when the order is decreased). However, in the particular case of $\delta=0$ the interlacing of $C_\nu(x)$ and $C_\mu(x)$ is preserved for $\nu,\mu>-1$.

Additional interlacing relations can be proved with the aid of the tools introduced below, e.g. between $J_{\nu+2}(x)$ and $J_{\nu}'(x)$, but only for specific differences between the orders (which is $2$ in this particular example), and thus not in the form of Theorem \ref{THM1}.

In order to prove Theorem~\ref{THM1} two tools are utilized. The first is the conditional transitivity of interlacing relations.

\begin{lemma}\label{lem:trans}
Let $f$, $g$ and $h$ be continuous functions on some common interval $I$. Suppose $f$ is interlaced with $g$ and $g$ is interlaced with $h$ on $I$, where
\begin{equation}\label{funeq}
a(x)f(x)+b(x)g(x)+c(x)h(x)=0
\end{equation}
with some functions $a$, $b$, $c$ satisfying $\sgn a(x)=const.\neq0$, $\sgn b(x)=const.\neq0$ and $\sgn c(x)=const.\neq 0$. Then $f$ is interlaced with $h$ on $I$.
\end{lemma}

The second tool is a result connecting Wronskians and interlacing.

\begin{lemma}\label{lem:WR}
The Wronskian $W\left(\sqrt{x}C_\nu(x),\sqrt{x}\bar C_\mu(x)\right)$ has no roots on the interval $x\in(\min (c_{\nu,1},\bar{c}_{\mu,1}),\infty)$ if and only if the positive zeros of the functions $C_\nu(x)$ and $\bar{C}_\mu(x)$ are interlaced.
\end{lemma}

\section{Proofs}

\subsection{Three term recurrence relations}

\begin{proof}[Proof of Lemma~\ref{lem:trans}]
Let $\{x_i\}$ and $\{y_i\}$ denote the sets of zeros of $f$ and $h$ on $I$, respectively. Then the functional equation (\ref{funeq}) yields $\sgn g(x_i)=-\sgn(bc)\sgn h(x_i)$ and $\sgn f(y_i)=-\sgn(ab)$ $\sgn g(y_i)$. Since $f$ and $g$ are interlaced we have $\sgn g(x_i)=-\sgn g(x_{i+1})$, similarly $\sgn g(y_i)=-\sgn g(y_{i+1})$. Then $h$ ($f$) must have an odd number of zeros between each consecutive pair of zeros of $f$ ($h$) implying the two are interlaced on the interval $I$.
\end{proof}

We prove two interlacing relations using Lemma~\ref{lem:trans}.

\begin{cor}\label{cor:inter1}
For $\nu>0$ the positive zeros of $C_\nu(x)$ and $C_{\nu+2}(x)$ are interlaced.
\end{cor}
\begin{proof}
Indeed, Lemma~\ref{lem:trans} yields the statement, since with $I=(0,\infty)$, $f=C_\nu$, $g=C_{\nu+1}$ and $h=C_{\nu+2}$ Eq. (\ref{funeq}) can be turned into
\begin{equation}
C_\nu(x)-\frac{2\nu+2}{x}C_{\nu+1}(x)+C_{\nu+2}(x)=0,
\end{equation}
which is a known three term recurrence relation.
\end{proof}

For the derivative functions a suitable three term recurrence relation can be found using the well-known ones \cite{Abramowitz}. From
\begin{align}
&C'_\nu(x)=-C_{\nu+1}(x)+\frac{\nu}{x}\,C_\nu(x),\qquad
&C'_{\nu+1}(x)=C_{\nu}(x)-C_{\nu+2}(x),\\
&C'_{\nu+1}(x)=C_{\nu}(x)-\frac{\nu+1}{x}\,C_{\nu+1}(x),\qquad
&C'_{\nu+2}(x)=C_{\nu+1}(x)-\frac{\nu+2}{x}\,C_{\nu+2}(x)
\end{align}
we infer that
\begin{equation}\label{eq:derrec}
[x^2-(\nu+1)(\nu+2)]C'_{\nu}(x)+[x^2-\nu(\nu+1)]C'_{\nu+2}(x)=\frac{2(\nu+1)}{x}[x^2-\nu(\nu+2)]C'_{\nu+1}(x)
\end{equation}
holds.

The first zero of $C'_\nu(x)$ can be at any point of the half line $(0,\infty)$ depending on $\nu$ and $\delta$. Eq. (\ref{eq:derrec}) implies that the first few zeros of $C'_\nu(x)$ and $C'_{\nu+2}(x)$ may not be interlaced even if $C'_\nu(x)$ and $C'_{\nu+1}(x)$ are interlaced. For $x>\sqrt{(\nu+1)(\nu+2)}$ $C_\nu'(x)$ and $C_{\nu+2}'(x)$ are interlaced if $C'_\nu(x)$ and $C'_{\nu+1}(x)$ are interlaced. However, the first few zeros of $C_\nu'(x)$ and $C_{\nu+1}'(x)$ might still not be interlaced. One can only guarantee interlacing of the derivative functions $C_\nu'(x)$, $C_{\nu+1}'(x)$ and $C_{\nu+2}'(x)$ if $\delta=0$ or $\delta=\frac{\pi}{2}$.

\begin{cor}\label{lem:inter2}
The positive zeros of $J'_\nu(x)$ and $J'_{\nu+2}(x)$ and those of  $Y'_\nu(x)$ and $Y'_{\nu+2}(x)$ are interlaced if $\nu>0$.
\end{cor}
\begin{proof}
The multiplying terms in Eq. (\ref{eq:derrec}) are all positive for $x>j'_{\nu+2,1}$ thus Lemma~\ref{lem:trans} yields that $J'_\nu$ and $J'_{\nu+2}$ are interlaced on $(j'_{\nu+2,1},\infty)$ since $J'_\nu$, $J'_{\nu+1}$ and $J'_{\nu+1}$, $J'_{\nu+2}$ are interlaced (Theorem~\ref{thm:1}). It remains to show that $J'_\nu$ has only one zero ($j'_{\nu,1}$) before $j'_{\nu+2,1}$. We have $j'_{\nu,1}<j'_{\nu+1,1}<j'_{\nu+2,1}$ from Theorem~\ref{thm:cont}, while Theorem~\ref{thm:1} implies one further zero ($j'_{\nu,2}$) on $(j'_{\nu+1,1},j'_{\nu+1,2})$. Analyzing the signs in Eq. (\ref{eq:derrec}) yields that this zero must be after $j'_{\nu+2,1}$.

The same reasoning holds for the second order derivative functions.
\end{proof}

The following is a simple corollary of Theorem~\ref{thm:cont}.

\begin{cor}\label{cor}
If $\nu>0$ then the previous interlacing relations remain to be true if the difference between the orders is $\varepsilon$ instead of $2$ with $0<\varepsilon\leq2$.
\end{cor}

\subsection{Wronskians}

To prove the negative parts of Theorem~\ref{THM1} I analyze Wronskians as in \cite{PA1}. Let
\begin{align}
&\xi_\nu=\sqrt{x}C_\nu(x)=\sqrt{x}[\cos\delta J_\nu(x)-\sin\delta Y_\nu(x)],\\ &\bar{\xi}_\mu=\sqrt{x}\bar{C}_\mu(x)=\sqrt{x}[\cos\bar\delta J_\mu(x)-\sin\bar\delta Y_\mu(x)],
\end{align}
which functions give rise to the Wronskian
\begin{equation}
W\left(\sqrt{x}C_\nu(x),\sqrt{x}\bar C_\mu(x)\right)\equiv W_{\xi_\nu,\bar{\xi}_\mu}(x)=\xi_\nu(x)\bar{\xi}'_{\mu}(x)-\xi'_{\nu}(x)\bar{\xi}_{\mu}(x).
\end{equation}
Differentiating with respect to $x$ one obtains
\begin{equation}
W'_{\xi_\nu,\bar{\xi}_\mu}(x)=\frac{\mu^2-\nu^2}{x^2}\xi_\nu(x)\bar{\xi}_\mu(x),\label{derW}
\end{equation}
which holds because of the differential equation
\begin{equation}
x^2\left[\frac{d^2}{dx^2}+1\right]\xi_\nu(x)=\left(\nu^2-\frac{1}{4}\right)\xi_\nu(x),
\end{equation}
inferred from the Bessel equation.

From Eq. (\ref{derW}) follows that the set of local extrema of $W_{\xi_\nu,\bar{\xi}_\mu}(x)$ is $\{w_{\nu\mu,s}\}_{s=1}^\infty=\{c_{\nu,s}\}_{s=1}^\infty\cup\{\bar{c}_{\mu,s}\}_{s=1}^\infty$. At these positions the Wronskian takes
\begin{equation}\label{extrW}
\text{extr}_s W_{\xi_\nu,\bar{\xi}_\mu}(x)\equiv W_{\xi_\nu,\bar{\xi}_\mu}(w_{\nu\mu,s})=\begin{cases}-\xi'_{\nu}(c_{\nu,t})\bar{\xi}_{\mu}(c_{\nu,t})\\ +\xi_\nu(\bar{c}_{\mu,t})\bar{\xi}'_{\mu}(\bar{c}_{\mu,t}),\end{cases}
\end{equation}
where the exact value of $t$ depends on the interlacing of $C_\nu(x)$ and $\bar{C}_\mu(x)$.

Now we are ready to prove Lemma~\ref{lem:WR}.

\begin{proof}[Proof of Lemma~\ref{lem:WR}]
Let $\bar{c}_{\mu,1}<c_{\nu,1}$. Without loss of generality suppose $\sgn C_\nu(0+)$ $=\sgn C_\mu(0+)=1$ (otherwise to satisfy the equation we take the opposite of the respective functions, whose zeros coincide with the original ones).
Please note that $\sgn C'_\nu(c_{\nu,n})=(-1)^n$.

Suppose the zeros of $C_\nu(x)$ and $C_\mu(x)$ are interlaced implying $\sgn C_\nu(\bar c_{\mu,n})=(-1)^{n+1}$ and $\sgn\bar C_\mu( c_{\nu,n})=(-1)^{n}$. Then every odd (even) numbered extremum is at a zero of $\bar C_\mu(x)$ ($C_\nu(x)$). From Eq. (\ref{extrW}) and the signs of the constituent functions it follows that
\begin{equation}
\sgn\text{extr}_n W_{\xi_\nu,\bar\xi_\mu}(x)=-1
\end{equation}
independent of $n$ implying for $W_{\xi_\nu,\bar\xi_\mu}(x)$ no zeros on $(\min (c_{\nu,1},\bar{c}_{\mu,1}),\infty)$.

Since $\{w_{\nu\mu,s}\}_{s=1}^\infty=\{c_{\nu,s}\}_{s=1}^\infty\cup\{\bar{c}_{\mu,s}\}_{s=1}^\infty$ it is apparent that the converse of the statement is true as well.
\end{proof}

This lemma will now be used to derive the breaking conditions (negative parts) for Theorem~\ref{THM1}.

In what follows I will use some asymptotic properties of the Bessel functions. From the definitions of $J_\nu(x)$ and $Y_\nu(x)$, i.e.
\begin{equation}
J_\nu(x)=\sum_{m=0}^\infty\frac{(-1)^m}{m!\ \Gamma(m+\nu+1)}\left(\frac{x}{2}\right)^{2m+\nu},\qquad Y_\nu(x)=\frac{J_\nu(x)\cos(\nu\pi)-J_{-\nu}(x)}{\sin(\nu\pi)}
\end{equation}
it is inferred that for $\nu>0$
\begin{equation}\label{C0}
C_\nu(x)=\sin\delta\left(\frac{\Gamma(\nu)2^\nu}{\pi}+o(1)\right)x^{-\nu},\qquad x\to0.
\end{equation}

The asymptotics of the Wronskian can be derived from the asymptotics of the Bessel functions, namely
\begin{align}\label{eq:Basy}
J_\nu(x)&=\sqrt{\frac{2}{\pi x}}\cos \left(x-\frac{\nu\pi}{2}-\frac{\pi}{4}\right)+o(1),\qquad x\to\infty,\\
Y_\nu(x)&=\sqrt{\frac{2}{\pi x}}\sin \left(x-\frac{\nu\pi}{2}-\frac{\pi}{4}\right)+o(1),\qquad x\to\infty,
\end{align}
therefore
\begin{equation}\label{eq:Was}
W_{\xi_\nu,\bar{\xi}_\mu}(x)=\frac{2}{\pi}\sin\left(\frac{\mu-\nu}{2}\pi+\delta-\bar\delta\right)+o(1),\qquad x\to\infty
\end{equation}
meaning that the Wronskian converges to a constant at infinity.

\begin{lemma}\label{lem:BD}
Let $\nu,\ \mu>0$. Then the interlacing of $C_\nu(x)$ and $\bar C_{\mu}(x)$ breaks down in the following cases:
\begin{align*}
a)&\bar{C}_\mu(x)\equiv C_\mu(x) \text{ with } |\nu-\mu|>2;\\
b)&\quad C_\nu(x)\equiv J_\nu(x)\text{ and }\bar{C}_\mu(x)\equiv Y_\mu(x) \text{ with } |\nu-\mu|>1\text{ provided that } y_{\mu,1}<j_{\nu,1}.
\end{align*}
\end{lemma}
\begin{proof}
a. The proof is elementary in view of Lemma~\ref{lem:WR}. One only needs to show that the Wronskian associated to the given cylinder functions has at least one zero between its first extremum and infinity.

From Eq. (\ref{C0}) it follows, that independently of $\nu$ $\sgn C_\nu(0+)=\sgn \sin\delta$. Let $\mu<\nu$. Since $\sgn C_\nu(0+)=\sgn C_\mu(0+)$ and $c_{\mu,1}<c_{\nu,1}$ the first extremum of $W_{\xi_\nu,\bar\xi_\mu}(x)$ is positive. (For $\sin\delta=0$ we have two Bessel functions of the first kind and $\sgn J_\nu(0+)=\sgn J_\mu(0+)$ still holds.)
In Eq. (\ref{eq:Was}) we have $\delta-\bar\delta=0$, thus if $4k<\nu-\mu<2+4k$ ($k\in\mathbb{Z}^+$) the Wronskian is positive at the first extremum and negative at infinity, which assumes an odd number of zeros on this interval. By Lemma~\ref{lem:WR} in this case $C_\nu(x)$ and $C_\mu(x)$ are not interlaced.

It is easy to see now that by increasing $\nu$ (to reach the uncovered regions of the previous argumentation) the interlacing is not recovered. Let $S>0$ be such that for $n<S$ $c_{\mu,n}<c_{\nu,n}<c_{\mu,n+1}$ but $c_{\mu,S}<c_{\nu,S}<c_{\mu,S+1}<c_{\mu,S+2}<c_{\nu,S+1}$, i.e. only the first $S$ zeros of $C_\nu(x)$ and $C_\mu(x)$ are interlaced. Because of Theorem~\ref{thm:cont} interlacing cannot be recovered by increasing $\nu$ ($c_{\mu,S+2}<c_{\nu,S+1}<c_{\nu+\varepsilon,S+1}$, $\forall\varepsilon>0$).

b.~(This part was already proven in \cite{PA1}; however, in a more complicated way.) Let $\mu<\nu$. In this case the first extremum is negative since $\sgn J_\nu(0+)=-\sgn Y_\mu(0+)$, while $\delta-\bar\delta=-\frac{\pi}{2}$ in Eq. (\ref{eq:Was}) implies for $1+4k<\nu-\mu<3+4k$ ($k\in\mathbb{Z}^+$) that the Wronskian converges to a positive number. Therefore the Wronskian has at least one zero. For the uncovered regions of $|\mu-\nu|>1$ the same kind of reasoning works as the one we used in case a.
\end{proof}

From the proof one can see that "shifted interlacing" occurs on every $(a,b)$ intervals where $W_{\xi_\nu,\xi_\mu}(x)$ has no zeros. By "shifted interlacing" we mean $c_{\mu,s}<c_{\nu,s+d}<c_{\mu,s}$ for $s=s_1,s_2,\ldots, s_n$ with some fixed $d\neq 0$ shift (ordinary interlacing is defined by $d=0$). Especially important is the interval $(z,\infty)$ with $z$ being the greatest zero of the Wronskian.

\begin{lemma}\label{lem:BDD}
Let $\nu,\ \mu>0$. Then the interlacing of $C'_\nu(x)$ and $C'_\mu(x)$ breaks down for $|\nu-\mu|>2$, either $C\equiv J$ or $C\equiv Y$.
\end{lemma}
\begin{proof}
Let $\mu<\nu$. Using Lemma~\ref{lem:BD}a and the recurrence relation
\begin{equation}\label{eq:rec}
C'_\nu(x)=-C_{\nu+1}(x)+\frac{\nu}{x}C_\nu(x)
\end{equation}
I will show that the interlacing
\begin{equation}
c'_{\mu,1}<c'_{\nu,1}<c'_{\mu,2}<\ldots
\end{equation}
is certainly broken for $|\nu-\mu|>2$.

From the recurrence relation (\ref{eq:rec}) we infer that the zeros of $C'_\nu(x)$ converge to those of $C_{\nu+1}(x)$, moreover they can be identified with one another since $C'_\nu(x)$ and $C_{\nu+1}(x)$ are interlaced (Theorem~\ref{thm:1}) and also the zeros of both functions are well separated asymptotically (see Eq. (\ref{eq:Basy})). That is either $c_{\nu+1,s}\approx c'_{\nu,s}$ or $c_{\nu+1,s}\approx c'_{\nu,s+1}$ for big $s$'s.

Let now $\nu=\mu+2+K$ with some $4k<K<2+4k$ ($k\in\mathbb{Z}^+$). Then the Wronskian $W_{\xi_{\nu+1},\xi_{\mu+1}}(x)$ has an odd number of zeros implying for the zeros of $C_{\nu+1}(x)$ and $C_{\mu+1}(x)$ shifted interlacing on $(z,\infty)$. Because of the asymptotic identification between $C'_\nu(x)$ and $C_{\nu+1}(x)$ the shifted interlacing, that is a broken interlacing, also holds for the zeros of $C'_\nu(x)$ (with perhaps a different threshold index).

For the uncovered regions of $2+4k<K<4+4k$ ($k\in\mathbb{Z}^+$) the same kind of argument works that was used in Lemma~\ref{lem:BD}.
\end{proof}

In summary, the combination of Lemma~\ref{lem:BD} and Corollary~\ref{cor} yields the first part of Theorem~\ref{THM1} while Lemma~\ref{lem:BDD} and Corollary~\ref{cor} gives the second part.
 
\bibliographystyle{plain}
\bibliography{cyl}

\end{document}